\documentclass[12pt]{amsart}
\usepackage{amssymb,amsfonts,euscript,mathrsfs,latexsym,amscd,amsmath}
\usepackage{epsf, color, graphicx, amsthm,amsopn,bbm, microtype, url,tabularx}
\usepackage[utf8]{inputenc}
\usepackage[arrow,matrix]{xy}
\usepackage{amssymb}
\usepackage{hyperref}
\definecolor{darkblue}{RGB}{0,0,160}
\hypersetup{
colorlinks,%
citecolor=black,%
filecolor=black,%
linkcolor=darkblue,%
urlcolor=darkblue
}

\usepackage[
text={444pt,575pt},
headheight=9pt,
centering
]{geometry}

\renewcommand{\subset}{\subseteq}  
\renewcommand{\supset}{\supseteq}
\newcommand{\defas}{\mathrel{\mathop{:}}=}   
\def\cocoa{{\hbox{\rm C\kern-.13em o\kern-.07em C\kern-.13em o\kern-.15em A}}}

\newcommand{\QQ}{\mathbb{Q}}
\newcommand{\NN}{\mathbb{N}}

\newcommand{\ZZ}{\mathbb{Z}}

\newcommand{\kk}{\mathbbm{k}}

\newcommand{\<}{\langle}
\renewcommand{\>}{\rangle}

\DeclareMathOperator*{\codim}{codim} 

\DeclareMathOperator{\cone}{cone} 
\DeclareMathOperator{\charac}{char}

\DeclareMathOperator{\diag}{diag} 
\def\ol#1{{\overline {#1}}}

\newcommand{\inD}[1][\relax]{\def\argone{#1}\def\temprelax{\relax}
  \ifx\argone\temprelax\right.\else\,\middle|#1\right.{}\fi}

\newtheorem{thm}{Theorem}[section]
\newtheorem{lemma}[thm]{Lemma}

\newtheorem{prop}[thm]{Proposition}

\theoremstyle{definition}
\newtheorem{example}[thm]{Example}
\newtheorem{remark}[thm]{Remark}
\newtheorem{defn}[thm]{Definition}

\begin{document}

\title{Veronesean almost binomial almost complete intersections}
\date{August 2016}

\author{Thomas Kahle}
\address{Otto-von-Guericke Universität\\ Magdeburg, Germany}
\urladdr{\url{http://www.thomas-kahle.de}}
\author{Andr\'e Wagner}
\address{Technische Universität Berlin\\ Berlin, Germany}
\urladdr{\url{http://page.math.tu-berlin.de/~wagner}}


%

\begin{abstract}
  The second Veronese ideal $I_n$ contains a natural complete
  intersection $J_n$ generated by the principal $2$-minors of a
  symmetric $(n\times n)$-matrix.  We determine subintersections of
  the primary decomposition of $J_n$ where one intersectand is
  omitted.  If $I_n$ is omitted, the result is the other end of a
  complete intersection link as in liaison theory.  These
  subintersections also yield interesting insights into binomial
  ideals and multigraded algebra.  For example, if $n$ is even, $I_n$
  is a Gorenstein ideal and the intersection of the remaining primary
  components of $J_n$ equals $J_n+\<f\>$ for an explicit polynomial
  $f$ constructed from the fibers of the Veronese grading map.
\end{abstract}

\maketitle

\section{Introduction}

Ideals generated by minors of matrices are a mainstay of commutative
algebra.  Here we are concerned with ideals generated by $2$-minors of
symmetric matrices.  Ideals generated by arbitrary minors of symmetric
matrices have been studied by Kutz~\cite{kutz1974cohen} who proved, in
the context of invariant theory, that the quotient rings are
Cohen--Macaulay.  Results of Goto show that the quotient ring is
normal with divisor class group $\ZZ_2$ and Gorenstein if the format
of the symmetric matrix has the same parity as the size of the
minors~\cite{goto1977divisor,goto1979gorensteinness}.  Conca extended
these results to more general sets of minors of symmetric
matrices~\cite{conca1994divisor} and determined Gröbner bases and
multiplicity~\cite{conca1994grobner}.

Here we are concerned only with the binomial ideal $I_n$ generated by
the $2$-minors of a symmetric $(n\times n)$-matrix.  This ideal cuts
out the second Veronese variety and was studied classically, for
example by Gröbner~\cite{groebner65:veronese}.  It contains a complete
intersection $J_n$ generated by the principal $2$-minors
(Definition~\ref{d:JnIdeal}).  Coming from liaison theory one may ask
for the ideal $K_n = J_n : I_n$ on the other side of the complete
intersection link via~$J_n$.  In this paper we determine~$K_n$.

\begin{example}\label{e:n3}
  Consider the ideal
  $J_3 = \<ad-b^2, af-c^2, df-e^2\> \subset \QQ[a,b,c,d,e,f]$
  generated by the principal $2$-minors of the symmetric matrix
  $\left(\begin{smallmatrix}
      a & b & c \\
      b & d & e \\
      c & e & f \\
  \end{smallmatrix}\right)$.
  It is easy to check, for example with \textsc{Macaulay2}~\cite{M2},
  that $J_3$ is a complete intersection and has a prime decomposition
  $J_3 = I_3\cap K_3$ where
  \[
    I_3 = J_3 + \<ae-bc, cd-be, ce-bf\>
  \]
  is the second Veronese ideal, generated by all $2$-minors, and
  \[
    K_3 = J_3 + \<ae+bc, cd+be, ce+bf\>
  \]
  is the image of $I_3$ under the automorphism of $\QQ[a,\dots,f]$
  that maps $b,c$, and $e$ to their negatives and the remaining
  indeterminates to themselves.  As a very special case of
  Theorem~\ref{thm:mysterious} we find that the generator $ae+bc$ is
  the generating function of the fiber
  \[
    \left\{u \in \NN^6 :
    \begin{pmatrix}
      2 & 1 & 1 & 0 & 0 & 0 \\
      0 & 1 & 0 & 2 & 1 & 0 \\
      0 & 0 & 1 & 0 & 1 & 2
    \end{pmatrix}\cdot u =
    \begin{pmatrix}
      2\\
      1\\
      1
    \end{pmatrix}
  \right\}
  \]
  of the $\ZZ$-linear map~$V_3$ that defines the fine grading of
  $\QQ[a,\dots,f]/I_3$.  For $n\ge 4$ the extra generators are not
  binomials anymore and $K_n$ is an intersection of ideals obtained
  from $I_n$ by twisting automorphisms (Definition~\ref{d:twisting}).
  In Example~\ref{e:n4}, for $n=4$, we find $K_4 = J_4 + \<p\>$ for
  one quartic polynomial $p$ with eight terms.
\end{example}

General results on liaison theory of ideals of minors of symmetric
matrices have been obtained by
Gorla~\cite{gorla2007g,gorla2010symmetric}.  Our methods rely on the
combinatorics of binomial ideals and since $K_n$ is not binomial we
cannot explore the linkage class more with the present method.
Instead we are motivated by general questions about binomial ideals
and their intersections.  For example,
\cite[Problem~17.1]{kahle11mesoprimary} asks when the intersection of
binomial ideals is binomial.  From the primary (in fact, prime)
decomposition of $J_n$ we remove $I_n$ and intersect the remaining
binomial prime ideals.  The result is not binomial.  If $n$ is even,
$K_n = J_n + \<p\>$ for one additional polynomial~$p$.  In the
terminology of~\cite{Berger1963}, $K_n$ is thus an \emph{almost
complete intersection}.  It is also \emph{almost binomial}, as it is
principal modulo its binomial
part~\cite[Definition~2.1]{kahle16:findingBin}.  If $n$ is odd, then
there are $n$ additional polynomials (Theorem~\ref{thm:mysterious}).
While these numbers can be predicted from general liaison theory, our
explicit formulas reveal interesting structures at the boundary of
binomiality and are thus a first step
towards~\cite[Problem~17.1]{kahle11mesoprimary}

We determine $K_n$ with methods from combinatorial commutative
algebra, multigradings in particular~(see \cite[Chapters~7
and~8]{miller05:_combin_commut_algeb}).  The principal observation
that drives the proofs in Section~\ref{sec:principal-ideal} is that
the Veronese-graded Hilbert function of the quotient
$\kk[\mathbf{x}]/J_n$ becomes eventually constant
(Remark~\ref{r:bigEnough}).  The eventual value of the Hilbert
function bounds the number of terms that a graded polynomial can have.
The extra generators of $K_n$ are the lowest degree polynomials that
realize the bound.  We envision that this structure could be explored
independently and brought to unification with the theory of toral
modules from~\cite{dickenstein08:_combin_binom_primar_decom}.  Our
results also have possible extensions to higher Veronese ideals as we
outline in Section~\ref{s:extensions}.

Denote by $c_{n} \defas \binom{n}{2}$ the entries of the second
diagonal in Pascal's triangle.  Throughout, let
$[n] := \{1,\dots, n\}$ be the set of the first $n$ integers.  The
second Veronese ideal lives in the polynomial ring
$\kk[\NN^{c_{n+1}}]$ in $c_{n+1}$ indeterminates over a field~$\kk$.
For polynomial rings and quotients modulo binomial ideals we use
monoid algebra notation (see, for instance,
\cite[Definition~2.15]{kahle11mesoprimary}).  We make no a-priori
assumptions on $\kk$ regarding its characteristic or algebraic
closure, although care is necessary in characteristic two.  The
variables of $\kk[\NN^{c_{n+1}}]$ are denoted $x_{ij}$, for
$i,j\in [n]$ with the implicit convention that $x_{ij} = x_{ji}$.  For
brevity we avoid a comma between $i$ and $j$.  We usually think about
upper triangular matrices, that is $i\le j$.  The Veronese ideal
$I_{n}$ is the toric ideal of the \emph{Veronese multigrading
$\NN V_n$}, defined by the $(n\times c_{n+1})$-matrix $V_n$ with
entries
\begin{equation*}
  (V_n)_{i,jk} \defas
  \begin{cases} 
    2 & \text{ if $i=j=k$},\\
    1 & \text{ if $i=j$, or $i=k$, but not both},\\
    0 & \text{ otherwise}.
  \end{cases}
\end{equation*}
That is, the columns of $V_n$ are the non-negative integer vectors of
length $n$ and weight two.  For $\mathbf{b}\in\NN V_n$, the
\emph{fiber} is
$V_n^{-1}[\mathbf{b}] = \{u \in \NN^{c_{n+1}}: V_n u = \mathbf{b}\}$.
Computing the $V_n$-degree of a monomial is easy: just count how often
each row or column index appears in the monomial.  For example,
$\deg(x_{12}x_{nn}) = (1,1,0,\dots,0,2)$.  We do not distinguish row
and column vectors notationally, in particular we write columns as
rows when convenient.  Gröbner bases for a large class of toric ideals
including $I_n$ have been determined by
Sturmfels~\cite[Theorem~14.2]{sturmfels96:_gr_obner_bases_and_convex_polyt}.
The \emph{Veronese lattice $L_{n} \subset \ZZ^{c_{n+1}}$} is the
kernel of~$V_{n}$.  The rank of $L_n$ is $c_{n}$ since the rank of
$V_n$ is $n$ and $c_{n+1} - n = c_{n}$.  Lemma~\ref{l:veroLat} gives a
lattice basis.  With $\{e_{ij}, i\le j \in [n]\}$ a standard basis of
$\ZZ^{c_{n+1}}$, we use the following notation
\[
  [ij|kl] := e_{ik} + e_{jl} - e_{il} - e_{jk} \in \ZZ^{c_{n+1}}.
\] 
Then $[ij|kl]$ is the exponent vector of the minor
$x_{ik}x_{jl} - x_{il}x_{jk}$.

\subsection{Acknowledgement}
\label{sec:acknowledgement}
The authors are grateful to Aldo Conca for posing the question of
determining the link of the Veronese variety through its complete
intersection of principal minors and valuable comments on an early
version of the manuscript.  During the initial stages of this work,
Thomas Kahle was supported by the research focus dynamical systems of
the state Saxony-Anhalt. Andr\'e Wagner's research is carried out in
the framework of {\sc Matheon} supported by Einstein Foundation
Berlin.

\section{Decomposing and Recomposing}
\label{sec:principal-ideal}

\begin{lemma}\label{l:veroLat} The set 
  \[ 
    \mathcal B=\{[in|jn] : i,j \in [n-1]\}
  \] 
  is a lattice basis of the Veronese lattice~$L_n$.
\end{lemma}
\begin{proof} Write the elements of $\mathcal{B}$ as the columns of a
  $(c_{n+1}\times c_n)$-matrix~$B$.  Deleting the rows corresponding
  to indices $(i,n)$ for $i\in [n]$ yields the identity matrix
  $I_{c_{n}}$.  Thus $\mathcal{B}$ spans a lattice of the correct rank
  and that lattice is saturated.  Indeed, the Smith normal form of $B$
  must equal the identity matrix $I_{c_{n}}$ concatenated with a zero
  matrix.  Thus the quotient by the lattice spanned by $\mathcal{B}$
  is free.
\end{proof}

The Veronese ideal contains a codimension~$c_{n}$ complete
intersection $J_{n}$ generated by the principal $2$-minors.
\begin{defn}\label{d:JnIdeal} 
  The \emph{principal minor ideal} $J_n$ is generated by all principal
  $2$-minors $x_{ii}x_{jj} - x_{ij}^2$ of a generic symmetric matrix.
  The \emph{principal minor lattice} $L'_n$ is the lattice generated
  by the corresponding exponent vectors $[ij|ij]$, $i,j\in [n]$.
\end{defn}

It can be seen that the principal minor lattice is minimally generated
by $[ij|ij]$.  It is an unsaturated lattice meaning that it cannot be
written as the kernel of an integer matrix, or equivalently, that the
quotient $\ZZ^{c_{n+1}}/L_n'$ has torsion.  Since there are no
non-trivial coefficients on the binomials in $J_n$,
Proposition~\ref{p:isMesoprime} below says that it is a lattice ideal
with lattice~$L'_n$.  The twisted group algebra in
\cite[Definition~10.4]{kahle11mesoprimary} is just a group algebra.
Its torsion subgroup is given in the following proposition.

\begin{prop}\label{p:torsion} 
  The principal minor lattice is minimally generated by
  \[ 
    \mathcal B'=\{2[in|jn] : i\neq j \in [n-1]\} \cup \{[in|in] : i \in [n-1] \}.
  \]
  Furthermore the group $L_n/L_n'$ is (isomorphic to)
  $(\ZZ/2\ZZ)^{c_{n-1}}$.
\end{prop}
\begin{proof}
  It holds that $2[in|jn]=[in|in]+[jn|jn]-[ij|ij]$ and the map which
  includes the span of the elements $[ij|ij]$ into $L'_n$ is
  unimodular.
  A presentation of the group can be read off the Smith normal form of
  the matrix whose columns are a lattice basis.  Since $\mathcal{B}'$
  is a basis of $L'_n$, the columns and rows can be arranged so that
  the diagonal matrix $\diag(2,\dots,2,1,\dots,1)$ with $c_{n-1}$
  entries $2$ is the top $(c_n \times c_n)$-matrix of the Smith normal
  form.  Any entry below a two is divisible by two and thus row
  operations can be used to zero out the the bottom part of the
  matrix.  This yields the Smith normal form.
\end{proof}

The difference between the basis in Definition~\ref{d:JnIdeal} and
$\mathcal B'$ is that transition matrix from $\mathcal B$ to
$\mathcal B'$ is diagonal. Thus it is easy to describe the fundamental
zonotope of~$\mathcal B'$.

If $\charac(\kk) = 2$, then $J_n$ is primary over~$I_n$.  In all other
characteristics one can see that the Veronese ideal $I_{n}$ is a
minimal prime and in fact a primary component of~$J_{n}$.  These
statements follow from \cite{eisenbud96:_binom_ideal} and are
summarized in Proposition~\ref{p:decompose} below.  Towards this
observation, the next proposition says that $J_n$ is a mesoprime
ideal.

\begin{prop}\label{p:isMesoprime} 
  $J_{n}$ is a mesoprime binomial ideal and its associated lattice
  is~$L'_n$.
\end{prop}
\begin{proof} According to \cite[Definition~10.4]{kahle11mesoprimary}
  we show that $J_n = \<x^{u^+} - x^{u^-}: u\in L'_n\>$, since the
  quotient by this ideal is the group algebra
  $\kk[\ZZ^{c_{n+1}}/L'_n]$.  By the correspondence between
  non-negative lattice walks and binomial ideals
  \cite[Theorem~1.1]{diaconis98:_lattic} we prove that for any
  $u = u^+ - u^- \in L'_n$, the parts $u^+,u^-\in\NN^{c_{n+1}}$ can be
  connected using \emph{moves} $[ij|ij]$ without leaving
  $\NN^{c_{n+1}}$.  

  The vectors $u^+,u^-$ can be represented by upper triangular
  non-negative integer matrices.  From Definition~\ref{d:JnIdeal} it
  is obvious that all off-diagonal entries of $u^+ - u^-$ are
  divisible by two.  Since
  \[ \<x^{u^+} - x^{u^-}: u\in L'_n\> : x_{ij} = \<x^{u^+} - x^{u^-}:
    u\in L'_n\> \]
  for any variable $x_{ij}$, we can assume that $u^{+}$ and $u^-$ have
  disjoint supports and thus individually have off-diagonal entries
  divisible by two.  Consequently the moves $[ij|ij]$ allow to reduce
  all off-diagonal entries to zero, while increasing the diagonal
  entries.  As visible from its basis, the lattice $L'_n$ contains no
  nonzero diagonal matrices and thus $u^+$ and $u^-$ have been
  connected to the same diagonal matrix.
\end{proof}

\begin{remark}\label{r:groupAlgebra}
  From Proposition~\ref{p:torsion} it follows immediately that the
  group algebra $\kk[\ZZ^{c_{n+1}}]/J_n\kk[\ZZ^{c_{n+1}}]$ is
  isomorphic to $\kk[\ZZ^n\oplus (\ZZ/2\ZZ)^{c_{n-1}}]$.  In
  particular $\kk[\NN^{c_{n+1}}]/J_n$ is finely graded by the monoid
  $\NN V_n \oplus (\ZZ/2\ZZ)^{c_{n-1}}$.
\end{remark}

\begin{defn}\label{d:twisting}
  A \emph{$\ZZ_{2}$-twisting} is a ring automorphism of a (Laurent)
  polynomial ring that maps the indeterminates either to themselves or
  to their negatives.
\end{defn}

The lattice points in a fundamental zonotope of $L_n'$ play an
important role in the following developments.  The most succinct way
to encode them is using their generating function (which in this case
is simply a polynomial in the Laurent ring $\kk[\ZZ^{c_{n+1}}]$).  The
explicit form, of course, depends on the coordinates chosen.  The next
lemma is immediate from the definition of~$\mathcal{B}'$.
\begin{lemma}\label{l:genZono} 
  Let $M=\{[in|jn] : i\neq j \in [n-1]\}$.  The generating function of
  the fundamental zonotope of $\mathcal B'$ is
  \[ 
    p_n=\prod_{m\in M}(x^{m}+1)=\prod_{m\in M}(x^{m^+}+x^{m^-})
  \]
\end{lemma}

It is useful for the further development to pick the second
representation of $p_n$ in Lemma~\ref{l:genZono} as a representative
of $p_n$ in polynomial ring~$\kk[\NN^{c_{n+1}}]$.  Its image in the
quotient $\kk[\NN^{c_{n+1}}]/J_n$ also has a natural representation.
The terms of $p_n$ can be identified with upper triangular integer
matrices which arise as sums of positive and negative parts of
elements of $M$.  A positive part of $[in|jn] \in M$ has entries $1$
at positions $(i,j)$ and $(n,n)$ while a negative part has two entries
$1$ in the last column, but not at~$(n,n)$.  Modulo the moves
$\mathcal{B'}$, any exponent matrix of a monomial of $p_n$ can be
reduced to have only entries $0$ or $1$ in its off-diagonal positions.

\begin{remark}\label{r:pnterms}
  A simple count yields that $p_n$ has $V_n$-degree
  $(n-2,\ldots,n-2,2c_{n-1})$.  In the natural representation of
  monomials of $p_n$ as integer matrices with entries $0/1$ off the
  diagonal, there is a lower bound for the value of the $(n,n)$ entry.
  To achieve the lowest value, one would fill the last column with
  entries $1$ using negative parts of elements of $M$, and then use
  positive parts (which increase~$(n,n)$).  For example, if $n$ is
  even, there is one term of~$p_n$ whose last column arises from the
  negative parts of $[1n|2n], [3n|4n], \dots, [(n-3)n|(n-2)n]$ and
  then positive parts of the remaining elements of~$M$.  If $n$ is
  odd, then there is one term of $p_n$, whose $n$-th column is
  $(1,\dots,1,\sigma_{n-1})$ for some value~$\sigma_{n-1}$.  In fact,
  since $|M| = c_{n-1}$, the lowest possible value of the $(n,n)$
  entry is $\sigma_{n-1} = c_{n-1}-\lfloor \frac{n-1}{2}\rfloor$.
\end{remark}

The primary decomposition of $J_n$ is given by~\cite[Theorem~2.1 and
Corollary~2.2]{eisenbud96:_binom_ideal}.
\begin{prop}\label{p:decompose}
  If $\charac(\kk) = 2$, the $J_n$ is primary over~$I_n$.  In all
  other characteristics, there exists $2^{c_{n-1}}$ $\ZZ_2$-twisting
  $\phi_i$ with $i\in[c_{n-1}]$, such that the complete intersection
  $J_{n}$ has prime decomposition
  \begin{equation}\label{eq:decompose}
    J_n=\bigcap_i\phi_i(I_{n}).
  \end{equation}

\end{prop}

\begin{thm}\label{thm:mysterious} 
  If $n$ is odd intersecting all but one of the components in
  \eqref{eq:decompose} yields
  \[ 
    \bigcap_{i\neq l}\phi_i(I_{n})=J_n+ \< \phi_l(p^+_{n,i}) : i\in [n] \>,
  \]
  where $p^+_{n,i} \in\kk[\NN^{c_{n+1}}]$ are homogeneous polynomials
  of degree~$\frac{(n-1)^2}{2}$ that are given as generating functions of
  the fibers $V_n^{-1}[(n-2,\dots,n-2) + e_i]$.  If $n$ is even, then
  the same holds for a single polynomial $p^+_{n}$ of degree
  $\frac{n(n-2)}{2}$, given as the generating function of
  $V_n^{-1}[(n-2,\dots,n-2)]$.
\end{thm}
The proof of Theorem~\ref{thm:mysterious} occupies the remainder of
the section after the following example.

\begin{example}\label{e:n4}
  The complete intersection $J_4$ is a lattice ideal for the
  lattice~$L'_4$.  In the basis $\mathcal{B'}$, it is generated by the
  six elements
  \[
    \{2[i4|j4] : i<j \in [3] \} \cup \{[i4|i4] : i \in [3]\}.
  \]
  Three of the six elements correspond to principal minors:
  \[x_{11}x_{44}-x_{14}^2,x_{22}x_{44}-x_{24}^2,x_{33}x_{44}-x_{34}^2\]
  The other elements give the binomials
  \[
    x_{12}^2x_{44}^2-x_{14}^2x_{24}^2,
    x_{13}^2x_{44}^2-x_{14}^2x_{34}^2,
    x_{23}^2x_{44}^2-x_{24}^2x_{34}^2
  \]
  These six binomials do not generate $J_4$, but $J_4$ equals the
  saturation with respect to the product of the
  variables~\cite[Lemma~7.6]{miller05:_combin_commut_algeb}. The
  $2^3 = 8$ minimal prime components of $J_n$ are obtained by all
  possible twist combinations of the monomials $\pm x_{14}x_{24}$,
  $\pm x_{14}x_{34}$, $\pm x_{24}x_{34}$.  Consider the mysterious
  polynomial
  \[
    p_4=(x_{12}x_{44}+x_{14}x_{24})(x_{13}x_{44}+x_{14}x_{34})(x_{23}x_{44}+x_{24}x_{34}),
  \]
  which is the generating function of the fundamental zonotope of
  $L_4'$ in the basis~$\mathcal{B}'$ and of $V_4$-degree~$(2,2,2,6)$.
  \emph{In the Laurent ring} $\kk[\ZZ^{10}]$, the desired ideal
  $J_4 : I_4$ equals~$J_4 + \<p_4\>$.  To do the computation in the
  polynomial ring, we need to saturate with respect to
  $\prod_{ij}x_{ij}$.  If $n$ is even, this saturation generates one
  polynomial, if $n$ is odd, it generates $n$ polynomials.  Here,
  where $n=4$, the ideal $J_4 : I_4$ is generated by $J_4$ and the
  single polynomial
  \begin{align*}
    p_4^+ & =x_{11}x_{22}x_{33}x_{44}+x_{11}x_{23}x_{24}x_{34}+ x_{13}x_{14}x_{22}x_{34}+ x_{12}x_{14}x_{24}x_{33} \\ 
          & + x_{13}x_{14}x_{23}x_{24}+ x_{12}x_{14}x_{23}x_{34}+ x_{12}x_{13}x_{24}x_{34}+x_{12}x_{13}x_{23}x_{44}.
  \end{align*}
  Modulo the binomials in $J_4$, the polynomial $p_4^+$
  equals~$p_4/x_{44}^2$ (Lemma~\ref{l:pnpnplus}).
\end{example}

As a first step towards the proof of Theorem~\ref{thm:mysterious} we
compute the monoid $Q$ under which $\kk[\NN^{c_{n+1}}]/J_n$ is finely
graded, meaning that its $Q$-graded Hilbert function takes values only
zero or one.  That is, we make Remark~\ref{r:groupAlgebra} explicit.
\begin{lemma}\label{l:fiberparam}
  Fix $\mathbf{b} \in \cone(V_n)$ for some $n$.  The equivalence
  classes of lattice points in the fiber $V_n^{-1}[\mathbf{b}]$,
  modulo the moves $\mathcal{B}'$, are in bijection with set of 0/1
  matrices $u\in\{0,1\}^{c_{n+1}}$ of the following form
  \begin{itemize}
  \item $u_{ii} = 0$, for all $i\in [n]$
  \item $u_{in} = 0$, for all $i\in [n]$
  \item $\mathbf{b} - V_n u \in \NN^n$.
  \end{itemize}
\end{lemma}
\begin{proof}
  Each equivalence class of upper triangular matrices has a
  representative whose off-diagonal entries are all either zero or
  one.  The bijection maps such an equivalence class to the $c_{n-1}$
  entries that are off-diagonal and off the last column.  To prove
  that this is a bijection it suffices to construct the inverse map.
  To this end, let $u$ satisfy the properties in the statement.  In
  each row $i=1,\dots,n$, there are two values unspecified: the
  diagonal entry and the entry in the last column.  Given
  $\mathbf{b}_i$, and using the choice of representative modulo
  $\mathcal{B}'$ whose last column entries is either 0 or~1, fixes the
  diagonal entry too by linearity.  Therefore the map is a bijection.
\end{proof}

\begin{remark}\label{r:bigEnough}
  If $\mathbf{b}_i \ge (n-2)$ for all $i\in[n]$, then any 0/1 upper
  triangular $(n-2)$-matrix is a possible choice for the off-diagonal
  off-last column entries of $u$ in Lemma~\ref{l:fiberparam}.  An
  upper triangular $(n-2)$-matrix has $c_{n-1}$ entries.  Thus all
  those fibers have equivalence classes modulo $\mathcal{B}'$ that are
  in bijection with $\{0,1\}^{c_{n-1}}$.  In particular, each of those
  fibers, has the same number of equivalence classes.
\end{remark}

\begin{remark}\label{r:toral}
  Remark~\ref{r:bigEnough} implies that in the $V_n$-grading,
  $\kk[\NN^{c_{n+1}}]/J_n$ is \emph{toral} as in
  \cite[Definition~4.3]{dickenstein08:_combin_binom_primar_decom}: its
  $V_n$-graded Hilbert function is globally bounded by~$2^{c_{n-1}}$.
\end{remark}

If $n$ is odd, then $(n-2,\dots,n-2) \notin \NN V_n$.  Therefore the
minimal (with respect to addition in the semigroup $\cone(V_n)$)
fibers that satisfy Remark~\ref{r:bigEnough} are
$(n-1,n-2,\dots,n-2), \dots, (n-2,\dots,n-2,n-1)$.  If $n$ is even,
there is only one minimal fiber.

For the proof of Theorem~\ref{thm:mysterious} it is convenient to work
in the quotient ring $\kk[\NN^{c_{n+1}}]/J_n$.  Since
$I_n \supset J_n$ and $I_n$ is finely graded by~$\NN V_n$, each
equivalence class is contained in a single fiber
$V_n^{-1}[\mathbf{b}]$ and each fiber breaks into equivalence classes.
The following definition sums the monomials corresponding to these
classes for specific fibers.

\begin{defn}\label{d:mysteriousP}
  The \emph{minimal saturated fibers} are the minimal fibers that
  satisfy Remark~\ref{r:bigEnough}.  The generating function of the
  equivalence class in a minimal saturated fiber is denoted by
  $p_{n,i}^+$.  That is
  \[
    p_{n,i}^+ = \sum_{\mathbf{a}\in
    V_{n}^{-1}[\mathbf{b}_i]/L'_n}\mathbf x^{\mathbf{a}} \in
    \kk[\NN^{c_{n+1}}]/J_n.
  \]
  where $\mathbf{b}_i := (n-2,\dots,n-2) + e_i$ if $n$ is odd and
  $\mathbf{b}_i = (n-2,\dots,n-2)$ if $n$ is even.
\end{defn}

If $n$ is even, Definition~\ref{d:mysteriousP} postulates only one
polynomial which is simply denoted $p_n^+$ when convenient.
Sometimes, however, it can be convenient to keep the indices.

\begin{remark}\label{r:LaurentVersion}
  The construction of a generating function of equivalence classes of
  elements of the fiber in Definition~\ref{d:mysteriousP} can be
  carried out for any fiber of $V_n$.  For the fiber
  $(n-2,\dots,n-2,2c_{n-1})$ we get the polynomial $p_n$ from
  Lemma~\ref{l:genZono}.
\end{remark}

The quantity $\sigma_{n-1} = c_{n-1}-\lfloor \frac{n-1}{2}\rfloor$
(that is $c_{n-1}-\frac{n-2}{2} = \frac{(n-2)^2}{2}$ for even~$n$, and
$c_{n-1}-\frac{n-1}{2} = \frac{(n-1)(n-3)}{2}$ for odd $n$) appeared
in Remark~\ref{r:pnterms} and shows up again in the next lemma: it
almost gives the saturation exponent when passing from the Laurent
ring to the polynomial ring.

\begin{lemma}\label{l:pnpnplus}
  As elements of $\kk[\NN^{c_{n+1}}]/J_n$, if $n$ is even then,
  $x_{nn}^{\sigma_{n-1}} p_{n,i}^+ = p_n$, and if $n$ is odd, then,
  $x_{nn}^{\sigma_{n-1}+1} p_{n,i}^+ = x_{in}p_n$.
\end{lemma}
\begin{proof}
  If $n$ is even, the product $x_{nn}^{\sigma_{n-1}} p_{n,i}^+$ has
  $V_n$-degree $(n-2,\dots,n-2,2c_{n-1})$.  If $n$ is odd, the degree
  of $x_{nn}^{\sigma_{n-1} + 1} p_{n,i}^+$ equals
  $(n-2,\dots,n-2,2c_{n-1}+1) + e_i$.  Now these products equal $p_n$
  if $n$ is even and $x_{in}p_n$ if $n$ is odd by
  Remarks~\ref{r:bigEnough} and~\ref{r:LaurentVersion}.
\end{proof}

\begin{lemma}\label{l:syzygy}
  If $n$ is odd, then for any triple of distinct indices
  $i,j,k\in[n]$, in $\kk[\NN^{c_{n+1}}]/J_n$ we have
  $x_{ij}p_{n,k}^+ = x_{jk}p_{n,i}^+$.
\end{lemma}
\begin{proof}
  By Proposition~\ref{p:isMesoprime}, the variables are
  nonzerodivisors on $\kk[\NN^{c_{n+1}}]/J_n$.  The multidegree of
  $p_{n,k}^+$ satisfies the conditions of Remark~\ref{r:bigEnough},
  thus there are bijections between the monomials of $x_{ij}p_{n,k}^+$
  and $x_{jk}p_{n,i}^+$.  Since all relations in $J_n$ are equalities
  of monomials, multiplication with a variable does not touch
  coefficients.
\end{proof}

The following lemma captures an essential feature of our situation.
Since the $V_n$-graded Hilbert function of $\kk[\NN^{c_{n+1}}]$ is
globally bounded, there is a notion of \emph{longest homogeneous
polynomial} as one that uses all monomials in a given $V_n$-degree.
If such a polynomial is multiplied by a term, it remains a longest
polynomial.

\begin{lemma}\label{l:oneDimensional}
  The $V_n$-graded Hilbert functions of the
  $\kk[\NN^{c_{n+1}}]/J_n$-modules, $\<p_n\>$ and $\<p^+_{n,i}\>$,
  $i=1,\dots,n$ take only zero and one as their values.
\end{lemma}
\begin{proof}
  We only prove the statement for $\<p_n\>$ since the same argument
  applies also to $\<p^+_{n,i}\>$.  The claim is equivalent to the
  statement that any $f \in \<p_n\>$ is a term (that is, a monomial
  times a scalar) times~$p_n$.  Let $f = g p_n$ with a
  $V_n$-homogeneous~$g$.  Let $t_1,\dots,t_s$ be the terms of~$g$.
  Since $p_n$ is the sum of all monomials of degree $\deg(p_n)$, and
  multiplication by a term does not produce any cancellation, the
  number of terms of $t_i p_n$ equals that of $p_n$.  By
  Remark~\ref{r:bigEnough}, the monomials in degree $\deg(t_i p_n)$
  are in bijection with the monomials in degree $\deg(p_n)$, and
  therefore all $t_i p_n$ are scalar multiples of the generating
  function of the fiber for $\deg(t_i p_n)$.
\end{proof}

\begin{lemma}\label{l:colonContain}
  For any $i\in[n]$,
  $\<p^+_{n,i}\> : \left(\prod_{ij}x_{ij}\right)^\infty = \<p_{n,k}^+
  : k \in [n]\>$.
\end{lemma}
\begin{proof}
  If $n$ is odd, the containment of $p_{n,k}^+$ in the left hand side
  follows immediately from Lemma~\ref{l:syzygy}.  If $n$ is even, it
  is trivial.  For the other containment, let $f$ be a
  $V_n$-homogeneous polynomial that satisfies $mf \in \<p_{n,i}^+\>$
  for some monomial~$m$.  We want $f \in \<p_{n,k}^+ : k \in [n]\>$.
  By Lemma~\ref{l:oneDimensional}, $mf = tp^+_{n,i}$ for some
  term~$t$.  Since $mf$ has the same number of terms as $f$ and also
  the same number of terms as $tp^+_{n,i}$, this number must be
  $2^{c_{n-1}}$.  By Remark~\ref{r:bigEnough}, the only
  $V_n$-homogeneous polynomials with $2^{c_{n-1}}$ terms are monomial
  multiples of the $p_{n,k}^+$ for $k\in[n]$.
\end{proof}

\begin{prop}\label{p:saturate}
  $(J_n + \<p_n\>) : \left(\prod_{ij} x_{ij}\right)^\infty = J_n +
  \<p_{n,j}^+, j\in [n]\>$.
\end{prop}
\begin{proof}
  Throughout we work in the quotient ring
  $S := \kk[\NN^{c_{n+1}}]/J_n$ and want to show
  \[\<p_n\> : \left(\prod_{ij} x_{ij}\right)^\infty = \<p_{n,j}^+, j
    \in [n]\>.\]
  Lemma~\ref{l:pnpnplus} gives the inclusion $\supseteq$, since it
  shows that, modulo $J_n$, a monomial multiple of $p_{n,i}^+$ is
  equal to either $p_n$ or $x_{in}p_n$ and thus lies in $\<p_n\>$.
  For the other containment let
  \[f \in \<p_n\>:\left(\prod_{ij} x_{ij} \right)^\infty,\]
  that is $mf \in \<p_n\>$ for some monomial $m$ in~$S$.  This implies
  $mf = gp_n$ for some polynomial $g\in S$.  By
  Lemma~\ref{l:pnpnplus}, $x_{in}mf = g' p_{n,i}^+$ for some
  $g' \in S$.  So, $x_{in}mf \in \<p_{n,i}^+\>$ and thus
  $f\in \<p_{n,i}^+\>:x_{in}m$.  Lemma~\ref{l:colonContain} shows that
  $f \in \<p_{n,k} : k\in [n]\>$.
\end{proof}

\newcommand{\laur}{\kk[\ZZ^{c_{n+1}}]} Having identified the minimal
saturated fibers, the longest polynomials, and computed the saturation
with respect to the variables $x_{ij}$, we are now ready to prove
Theorem~\ref{thm:mysterious}.
\begin{proof}[Proof of Theorem~\ref{thm:mysterious}]
  After a potential renumbering, assume $\phi_1$ is the identity.  It
  suffices to prove the theorem for the omission of the Veronese ideal
  $i=1$ from the intersection.  The remaining cases follow by
  application of $\phi_l$ to the ambient ring.

  Consider the extensions $J_n\laur$ and $I_n\laur$ to the Laurent
  polynomial ring.  By the general Theorem~\ref{t:laurentIntersection}
  \[
    \bigcap_{i\neq 1} \phi_i(I_{n}\laur)=J_n\laur + \<p_n\>.
  \]  
  Pulling back to the polynomial ring, we have
  \[
    \bigcap_{i\neq 1} \phi_i(I_{n}) =(J_n + \<p_n\>) :
    (\prod_{x_{ij}}x_{ij})^\infty.\]
  Contingent on Theorem~\ref{t:laurentIntersection}, the result now
  follows from Proposition~\ref{p:saturate}.
\end{proof}

We have reduced the proof of Theorem~\ref{thm:mysterious} to a general
result on intersection in the Laurent ring.  It is a variation
of~\cite[Theorem~2.1]{eisenbud96:_binom_ideal}.  According to
\cite[Section~2]{eisenbud96:_binom_ideal}, any binomial ideal in the
Laurent polynomial ring $\kk[\ZZ]$ is defined by its lattice
$L\subset\ZZ^n$ of exponents and a partial character
$\rho : L \to \kk^*$.  Such an ideal is denoted $I(\rho)$ where the
lattice $L$ is part of the definition of~$\rho$.

\begin{thm}\label{t:laurentIntersection} 
  Let $\kk$ be a field such that $\charac(k)$ is either zero or does
  not divide the order of the torsion part of~$\ZZ^n/L$ and
  $I(\rho) \subset \kk[\ZZ^n]$ be binomial.  Let
  $I(\rho)=I(\rho'_1)\cap \ldots \cap I(\rho'_k)$ be a primary
  decomposition of $I(\rho)$ over the algebraic closure $\ol\kk$
  of~$\kk$.  Omitting one component $I(\rho'_{i^*})$ yields
  \[
    \bigcap_{i \neq i^*} I(\rho'_i)=I(\rho)+\rho'_{i^*}(p_L)
  \]
  where $p_{L}$ is the generating function of a fundamental zonotope
  of the lattice~$L$.
\end{thm}
\begin{proof}
  A linear change of coordinates in $\ZZ^{n}$ corresponds to a
  multiplicative change of coordinates in $\kk[\ZZ^{n}]$.  Since the
  inclusion of $L\subset\ZZ^{n}$ can be diagonalized using the Smith
  normal form, one can reduce to the case that $I(\rho)$ is generated
  by binomials $x_i^{q_i} - 1$.  This case follows by multiplication
  of the results in the univariate case.  The univariate case, in
  turn, is up to scaling given by the polynomials $(x^n-1)/(x-1)$.
\end{proof}

The assumption on $\charac(k)$ in Theorem~\ref{t:laurentIntersection}
can be relaxed at the cost of a case distinction similar to that in
\cite[Theorem~2.1]{eisenbud96:_binom_ideal}.

The explicit form of $p_L$ depends on a choice of lattice
basis. Because the notions lattice basis ideal and lattice ideal are
not the same in the polynomial ring (they are in the Laurent
polynomial ring), one needs to pull back using colon ideals to get a
result in the polynomial ring.  Even if in the Laurent ring the
subintersection in Theorem~\ref{t:laurentIntersection} is principal
modulo $I(\rho)$, it need not be principal in the polynomial ring (as
visible in Theorem~\ref{thm:mysterious}).  It would be very nice to
find more effective methods for binomial subintersections in the
polynomial ring, but at the moment the following remark is all we
have.
\begin{remark}\label{r:satPoly}  Under the field assumptions in
  Theorem~\ref{t:laurentIntersection}, let $I \subset \kk[\NN^n]$ be a
  lattice ideal in a polynomial ring with indeterminates
  $x_1,\dots,x_n$.  There exists a partial character
  $\rho : L \to \kk^*$ such that $I = I(\rho) \cap \kk[\NN^n]$.  The
  intersection of all but one minimal primary component of $I$ is
  \[
    (I(\rho)+\rho(p_L)) \cap \kk[\NN^n] = (I + \rho(p)m) :
    (\prod_{i=1}^n x_i)^\infty.
  \]
  where $p_L$ is the generating function of a fundamental zonotope
  of~$L$, and $m$ is any monomial such that
  $\rho(p_L)m \in \kk[\NN^n]$.
\end{remark}

\section{Extensions}
\label{s:extensions}\enlargethispage{3ex}

The broadest possible generalization of the results in
Section~\ref{sec:principal-ideal} may start from an arbitrary toric
ideal~$I \subset \kk[\NN^n]$, corresponding to a grading
matrix~$V\in\NN^{d\times n}$, and a subideal $J\subset I$, for example
a lattice basis ideal.  One can then ask when the quotient
$\kk[\NN^n]/J$ is toral in the grading~$V$.  The techniques in
Section~\ref{sec:principal-ideal} depend heavily on this property and
the very controllable stabilization of the Hilbert function.  One can
get the feeling that this happens if $J\subset I$ is a lattice ideal
for some lattice that is of finite index in the saturated lattice
$\ker_\ZZ(V)$.  However, such a $J$ cannot always be found: by a
result of Cattani, Curran, and Dickenstein, there exist toric ideals
that do not contain a complete intersection of the same
dimension~\cite{cattani2007complete}.

A more direct generalization of the results of
Section~\ref{sec:principal-ideal} was suggested to us by Aldo Conca.
The $d$-th Veronese grading $V_{d,n}$ has as its columns all vectors
of length $n$ and weight~$d$.  The corresponding toric ideal is the
$d$-th Veronese ideal $I_{d,n} \subset S = \kk[\NN^N]$ and it contains
a natural complete intersection~$J_{d,n}$ defined as follows.  The set
of columns of $V_{d,n}$ includes the multiples of the unit vectors
$D := \{de_i, i=1,\dots,n\}$.  For any column $v \notin D$, let
$f_v = x_v^d - \prod_i x_{de_i}^{v_i}$.  Then
$J = \<f_v : v\notin D\> \subset I_{d,n}$ is a complete intersection
with $\codim (J_{d,n}) = \codim(I_{d,n})$.  It is natural to
conjecture that a statement similar to Proposition~\ref{p:isMesoprime}
is true.  In this case, however, the group $L/L'$
(cf.~Proposition~\ref{p:torsion}) has higher torsion.  This implies
that the binomial primary decomposition of $J$ exists only if $\kk$
has corresponding roots of unity.
By results of Goto and Watanabe \cite[Chapter~3]{goto1978graded} on
the canonical module (cf.~\cite[Exercise~3.6.21]{winfried1998cohen})
the ring $S/I$ is Gorenstein if and only if $d | n$, so that $J:I$ is
equal to $J + (p)$ for some polynomial $p$ exactly in this situation.

In Section~\ref{sec:principal-ideal}, the notation can be kept in
check because there is a nice representation of monomials as upper
triangular matrices (Proposition~\ref{p:isMesoprime},
Lemma~\ref{l:fiberparam}, etc.).  To manage the generalization, it
will be an important task to find a similarly nice representation.  It
is entirely possible that something akin to the string notation
of~\cite[Section~14]{sturmfels96:_gr_obner_bases_and_convex_polyt}
does the job.  Additionally, experimentation with
\textsc{Macaulay2}---which has informed the authors of this
paper---will be hard.  For example, for $d=3, n=4$, the group $L/L'$
from Proposition~\ref{p:torsion} is isomorphic to $(\ZZ/3\ZZ)^{13}$
which means that a prime decomposition of $J_{3,4}$ has 1594323
components.  Computing subintersections of it is out of reach.  It may
be possible to compute a colon ideal like $(J_{3,4}:I_{3,4})$
directly, but off-the-shelf methods failed for us.

\bibliographystyle{amsplain} 
\bibliography{math}

\end{document}